\documentclass[10pt]{amsart}

\usepackage{amsmath,amssymb,verbatim}
\usepackage{amsthm}
\usepackage{xcolor}
\usepackage{enumerate}
\usepackage{url}
\usepackage[letterpaper, portrait, margin=1in]{geometry}

\usepackage{fullpage}
\usepackage{orcidlink}

\newenvironment{altproof}{\begin{proof}[\bfseries Alternative Proof of Proposition \ref{prop:order}.]}{\end{proof}}

\newtheorem{theorem}{Theorem}[section]
\newtheorem{fact}[theorem]{Fact}
\newtheorem{lemma}[theorem]{Lemma}
\newtheorem{corollary}[theorem]{Corollary}
\newtheorem{proposition}[theorem]{Proposition}

\theoremstyle{definition}
\newtheorem{definition}[theorem]{Definition}

\newtheorem{remark}[theorem]{Remark}

\def\Def{\operatorname{Def}}

\def\Av{\operatorname{Av}}

\def\conv{\operatorname{conv}_{\delta}}
\def\opp{\operatorname{opp}}
\def\Ult{\operatorname{Ult}}
\def\tp{\operatorname{tp}}
\def\co{\overline{\operatorname{co}}}

\title{Measures and Stability in a Model, revisited}
\keywords{Model theory, Keisler measures, stability in a model, stability, Morley product, double limit, randomization, VC theory.}
\author[C. D'Elb\'{e}e]{Christian D'Elb\'{e}e$^{*}$}
\thanks{$^{*}$ Fully supported by the Ramon y Cajal grant RYC2023-042677-I funded by MICIU/AEI/10.13039/501100011033 and by ESF+.}
\address{$^{*}$University of the Basque Country, Department of Mathematics (Leioa) / Institute for Logic, Cognition, Language and Information (Donostia-San Sebastián), Spain;\orcidlink{0000-0002-0268-1802}}
\email{christian.delbee@ehu.eus} 

\author[K. Gannon]{Kyle Gannon$^{\dagger}$}
\thanks{$^{\dagger}$ Partially supported by the Chateaubriand Fellowship of the Office for Science and
Technology of the Embassy of France in the United States; by the Fundamental Research Funds for the Central Universities, Peking University, grant no. 7100604835; by the National Natural Science Fund of China, grant no. 12501001}
\address{$^{\dagger}$Beijing International Center for Mathematical Research (BICMR) \\ 
Peking University \\ 
Beijing, China.\orcidlink{0000-0001-5951-5269}}
\email{kgannon@bicmr.pku.edu.cn}
\date{\today}
\begin{document}

\begin{abstract} This article is written in celebration of the 8th Kazakh-French Logical Colloquium. We expand on an unpublished research note of the second author. We record some results concerning local Keisler measures with respect to a formula which is \emph{stable in a model}. We prove that in this context, every local Keisler measure on the associated local type space is a weighted sum of (at most countably many) types. Using this observation, we give an elementary proof of the commutativity of the Morley product in this context.  We then give a functional analytic proof that the double limit property lifts to the appropriate evaluation map on pairs of local measures. We end with some comments on the NOP and local measures in the (properly) stable context. 
\end{abstract}

\maketitle

\section{Introduction}

The Franco--Kazakh connections in mathematical logic date back to the late 1980s with the collaboration between Tolendi Mustafin and Bruno Poizat -- leading to the first Soviet--French colloquium in model theory at Karaganda State University in Kazakhstan in 1990. Since then, the will to maintain and improve those mathematical links has persisted and were rekindled in recent times in Lyon in 2022, and then in Astanna in 2025. The reach of those connections naturally go far beyond France and Kazakhstan, as they witness fruitful transfer of knowledge and cross-fertilization of ideas among researchers from Europe, Russia, the United States, and China. The present paper embodies this transversality, and present work stemming (in part) from the participation of the authors to the \textit{8th Kazakh-French Logical Colloquium}. The authors are very grateful to the organizers of that meeting. They hope that their contribution to the Kazakh mathematical library will help strengthen the preexisting intellectual and social relationships between our research groups.

After Ben Yaacov's original article connecting stability theory with some of Grothendieck's functional analytic work \cite{BEN}, the concepts of \emph{stability in a model} and \emph{NIP in a model} were studied by a myriad of researchers in the field \cite{conant2019note,Conant,khanaki2020stability,khanaki2018remarks,khanaki2024grothendieck,Pillay,STAR}. In particular, it was the subject of several intense discussions at the Notre Dame model theory seminar in the Spring of 2017. This research arose from that localized frenzy of activity. 


It turns out that fundamental results in local stability theory can be generalized to the context of \emph{stable in a model} -- and in particular, via reinterpreting some of Grothendieck's work on functional analysis. From our perspective, it is natural to ask, \emph{What does the theory of Keisler measures look like in this setting?} We show that it also closely resembles the picture in the stable context. First, we show that if $\varphi(x;y)$ is stable in $M$, then every $\varphi-$measure on $M$ (finitely additive probability measure on $\varphi-$definable subsets of $M^{x}$) is the sum of (at most countably many) weighted $\varphi-$types. As consequence, all $\varphi-$measures on $M$ are $\varphi^{\opp}-$definable. Thus evaluating the Morley product between arbitrary pairs of $\varphi$ and $\varphi^{\opp}-$measures on the formula $\varphi(x,y)$ is well-defined. In \cite{BEN}, Ben Yaacov demonstrates that the \emph{fundamental theorem of stability theory} extends to the \emph{stable in a model} context; it then follows from the observations above that the Morley product evaluated at $\varphi(x,y)$ commutes on appropriate pairs of measures. In other words, if $\varphi(x;y)$ is stable in $M$, $\mu \in \mathfrak{M}_{\varphi}(M)$, and $\nu \in \mathfrak{M}_{\varphi^{\opp}}(M)$, then
\begin{equation*} 
(\mu_{x} \otimes \nu_{y})(\varphi(x,y)) = (\nu_{y} \otimes \mu_{x})(\varphi(x,y)). 
\end{equation*} 

The final object of interest is the evaluation map itself. We consider the function $E_{\varphi}:\mathfrak{M}_{\varphi}(M) \times \mathfrak{M}_{\varphi^{\opp}}(M) \to [0,1]$ via, 
\begin{equation*} 
E_{\varphi}(\mu,\nu) = (\mu_{x} \otimes \nu_{y})(\varphi(x,y)). 
\end{equation*} 
We show that if $\varphi(x;y)$ is stable in $M$, then $E_{\varphi}$ also witnesses the appropriate variant of the \emph{non-order property}. This follows more or less directly from results in functional analysis, i.e., Grothendeick's double limit theorem and the Krien-Smulian theorem. Finally, we consider the context where $\varphi(x;y)$ is a stable formula (i.e., $\varphi(x;y)$ does not have the $k-$order property for some fixed $k$). We prove that the map $E_{\varphi}$ is $(r,\epsilon)-$stable for any choice of $r \in (0,1)$ and $\epsilon > 0$. We remark that this follows implicitly by a result of Ben Yaacov and Keisler, namely the fact that the randomization of a stable formula remains stable (\cite[Theorem 5.14]{ben2009randomizations}). We exposit why this is true and then provide a different proof which follows from the VC-theorem, mimicking techniques involving measures in NIP theories (i.e., See the proof of \cite[Theorem 3.12]{gannon2019local}). This latter proof implies the existence of bounds, yet we leave analysis along these lines open. 

\subsection{Preliminaries} Throughout the article, fix a language $\mathcal{L}$ and an $\mathcal{L}-$structure, $M$. We use the letters $x,y,z,\dots$ to denote finite tuples of variables. The formula $\varphi(x;y)$ is a partitioned $\mathcal{L}-$formula with \emph{variables} variables $x$ and \emph{parameter} variables $y$. We let $\varphi^{\opp}(y;x)$ be the same formula as $\varphi(x;y)$, but with exchanged roles for the variables and parameters. We let $S_\varphi(M)$ be the space of $\varphi-$types with parameters from $M$. We let $\Def_{\varphi}(M)$ be the Boolean algebra of definable subsets of $M$ generated by $\{\varphi(x,b): b\in M\}$. We will routinely identify definable sets with the formulas which define them. A $\varphi-$formula is an element of $\Def_{\varphi}(M)$. Likewise, we have analogous definitions for $S_{\varphi^{\opp}}(M)$ and $\Def_{\varphi^{\opp}}(M)$. A $\varphi^{\opp}-$definition for a type $p$ in $S_{\varphi}(M)$ is a $\varphi^{\opp}-$formula, $d^{\varphi^{\opp}}_p(y)$, such that for each $b \in M^{y}$, $\varphi(x,b) \in p$ if and only if $M \models d^{\varphi^{\opp}}_p(b)$. Finally, we let $\mathfrak{M}_{\varphi}(M)$ and $\mathfrak{M}_{\varphi^{\opp}}(M)$ denote the spaces of finitely additive probability measures on $\Def_{\varphi}(M)$ and $\Def_{\varphi^{\opp}}(M)$ respectively. We recall that we can identify a measure in each of these spaces canonically with a regular Borel probability measure on their corresponding type space, e.g. $\mathfrak{M}_{\varphi}(M)$ is in canonical correspondence with regular Borel probability measures on $S_{\varphi}(M)$. 

\begin{definition}[Double Limit Property] Let $X$ and $Y$ be sets and let $f:X \times Y \to [0,1]$. We say that $f$ has the \textit{double limit property} if for any two sequence $(a_i)_{i \in \mathbb{N}}$, $(b_j)_{j \in \mathbb{N}}$ with $a_i \in X$ and $b_j \in Y$,
\begin{equation*}
    \lim_i \lim_j f(a_i,b_j) = \lim_j \lim_i f(a_i,b_j)
\end{equation*}
provided limits on both sides exist. 
\end{definition} 

The definition of \emph{stable in a model} given in \cite{BEN} restricted to discrete structures is as follows:

\begin{definition} A formula $\varphi(x;y)$ is \textit{stable in $M$} if $\varphi:M^{x} \times M^{y} \to \{0,1\}$ has the double limit property, where $\varphi(a,b) = 1$ if $M \models \varphi(a,b)$ and $\varphi(a,b) = 0$ otherwise. 
\end{definition} 

We first remark that clearly $\varphi(x;y)$ is stable in $M$ if and only if $\varphi^{\opp}(x;y)$ is stable in $M$. We also remark that if $\varphi(x;y)$ is stable, then $\varphi(x;y)$ is stable in any model of the underlying theory. An example of a formula which is stable in a model but not stable is the edge relation in the graph constructed by taking the disjoint union of all finite graphs on subsets of $\mathbb{N}$.  

In \cite{BEN}, Ben Yaacov established a surprising connection between functional analysis and local stability. In particular, he gave a proof of the \textit{fundamental theorem of stability} using Grothendieck's \emph{double limit theorem} \cite{GR}. Before stating the theorem, let us briefly recall some basic functional analysis. \begin{definition}[Weak Topology] Let $Y$ be a Banach space over a field $F$ $(F = \mathbb{R}$ or $\mathbb{C}$). Let $Y^{*}$ be the space of continuous linear functionals from $Y$ to $F$. Then, the \textit{weak topology on} $Y$ is the coarsest topology such that for each element of $Y^{*}$ remains a continuous function from $Y$ to $F$. 
\end{definition} 

\begin{definition}[Relatively Weakly Compact] Let $Y$ be a Banach space and let $A \subset Y$. We say that $A$ is \textit{weakly compact} if $A$ is a compact subset of $Y$ under the weak topology. Furthermore, we say that $A$ is \textit{relatively weakly compact} if the closure of $A$ under the weak topology is weakly compact. 
\end{definition}

Let $X$ be a topological space. Then $C_b(X)$ denotes the Banach space of bounded, continuous, complex-valued functions on $X$, equipped with the uniform norm, $||\cdot||_{\infty}$. We say that a set $A$ is $||\cdot||_{\infty}-$bounded if there exists $c$ in $\mathbb{R}$ such that for all $f$  in $A$, $||f||_{\infty} < c$. Grothendieck's theorem is as follows: 
\begin{theorem}[Grothendieck \cite{GR}]\label{GR}
Let $X$ be an arbitrary topological space, $X_0 \subseteq X$ a dense susbset. Then the following are equivalent for $A \subset C_b(X)$: 
\begin{enumerate}[$(i)$] 
\item The set $A$ is relatively weakly compact in $C_b(X)$. 
\item The set $A$ is $||\cdot||_\infty -$bounded, and whenever $f_n \in A$ and $x_n \in X_0$ form two sequences we have that 
$$ \lim_n \lim_m f_n(x_m) = \lim_m \lim_n f_n(x_m),$$
provided both limits exist. 
\end{enumerate} 
\end{theorem}

Via the double limit theorem above, Ben Yaacov derived the following (among other results):

\begin{theorem}\label{B} Assume that $\varphi(x;y)$ is stable in $M$, $p \in S_{\varphi}(M)$, and $q \in S_{\varphi^{\opp}}(M)$. Then $p$ has a $\varphi^{\opp}-$definition $d^{\varphi}_{p}(y)$, $q$ has a $\varphi-$definition $d^{\varphi^{\opp}}_{q}(x)$, and $d^{\varphi^{\opp}}_p(y) \in q$ if and only if $d^{\varphi}_{q}(x) \in p$. 
\end{theorem}

Finally, we recall the following equivalence from Starchenko's research note on the topic.

\begin{fact}\label{Equiv} The following are equivalent: 
\begin{enumerate}[($i$)]
\item $\varphi(x;y)$ is stable in $M$. 
\item There does not exist $(a_i,b_j)_{(i,j) \in \omega \times \omega}$ from $M^{x} \times M^{y}$such that 
\begin{enumerate} 
\item either for every $i \neq j$, $M \models \varphi(a_i,b_j)$ if and only if $i < j$. 
\item or for all $i \neq j$, $M \models \varphi(a_i,b_)$ if and only if $i > j$. 
\end{enumerate} 

\item The map $\chi_{\varphi}:S_{\varphi}(M) \times S_{\varphi^{\opp}}(M) \to \{0,1\}$ has the double limit property where 
\begin{equation*} 
\chi_{\varphi}(p,q) = 1 \iff d_{p}^{\varphi}(y) \in q \iff d_{q}^{\varphi^{\opp}}(x) \in p. 
\end{equation*} 
\end{enumerate} 
\end{fact}

In the above fact, the equivalence of $(i)$ and $(ii)$ \cite[Lemma 1.3]{STAR}. Clearly, $(iii)$ implies $(i)$ and $(i)$ implies $(iii)$ follows from \cite[Theorem 1.5]{STAR} together with Theorem \ref{B} (Also see \cite{Pillay}).

\section{Local Keisler measures}

In this section, we prove that if $\varphi(x;y)$ is stable in $M$, then all $\varphi-$measures are (at most) countable sums of weighted $\varphi-$types. The proof of this theorem uses the Sobczyk--Hammer decomposition theorem for positive, bounded charges. We recall this theorem in the case of finitely additive probability measures. But first, we need to recall two different kinds of finitely additive measures. 

\begin{definition} Let $\mathbb{B}$ be a Boolean algebra of subsets of $X$ (containing both $X$ and $\emptyset$) and $\mu$ be a finitely additive probability measure on $\mathbb{B}$.
\begin{enumerate}
    \item We say that $\mu$ is \textit{strongly continuous} on $\mathbb{B}$ if for all $\epsilon > 0$ there exist $F_1,...,F_n \in \mathbb{B}$ such that $\{F_i\}_{i=1}^n$ forms a partition of $X$ and for each $i \leq n$, $\mu(F_i) < \epsilon$.
    \item We say that $\mu$ is \textit{$\{0,1\}-$valued} on $\mathbb{B}$ if for every $F$ in $\mathbb{B}$, $\mu(F) \in \{0,1\}$.
\end{enumerate}
\end{definition}

We refer the reader to \cite[Theorem 5.2.7]{Charge} for a proof of the following theorem.

\begin{theorem}[Sobczyk--Hammer Decomposition Theorem \cite{SH}] Let $\mathbb{B}$ be a Boolean algebra on $X$ (containing $\emptyset$ and $X$) and $\mu$ be a finitely additive probability measure on $\mathbb{B}$. Then, there exists an initial segment $I$ of $\mathbb{N}$, a sequence of distinct finitely additive probability measures $(\mu_{i})_{i \in I}$, and a sequence of non-negative real numbers $(r_i)_{i \in I}$, with the following properties,
\begin{enumerate}[($i$)]
    \item $\mu_0$ is strongly continuous on $\mathbb{B}$,
    \item $\mu_i$ is $\{0,1\}-$valued on $\mathbb{B}$ for every $i\geq 1$,
    \item $\sum_{i \in I} r_i = 1$, and 
    \item $\mu = \sum_{i \in I} r_i \mu_i$. 
\end{enumerate}

Further, the decomposition in $(iv)$ is unique (obviously, up to permutation of the sequence and non-trivially weighted measures (i.e., $r_i > 0$)). 
\end{theorem}

The Sobczyk--Hammer decomposition theorem allows us to decompose any finitely additive probability measure into a single strongly continuous measure and a sum of (at most countably many) $\{0,1\}-$valued measures. We will show that if $\varphi(x;y)$ is stable in $M$, then there do not exist any strongly continuous measures on $\Def_{\varphi}(M)$. Thus every finitely additive probability measure will be the ``weighted sum" of at most countably many types. 

\subsection{Measures are sums of types}

\begin{definition} Let $\mathbb{B}$ be a Boolean algebra on a set $X$. We say that $\mathbb{B}$ has a \textit{$2-$tree} if there exists $ T \in \mathcal{P}(\mathbb{B})$ such that $(T, \supsetneq)$ is an infinite, complete, binary tree, and if $A, C \in T$, $A \not \supset C$, and $C \not \supset A$, then $A \cap C = \emptyset$.
\end{definition} 

\begin{fact}\label{ultra} Let $\mathbb{B}$ be a Boolean algebra on a set $X$ and assume that $\mathbb{B}$ has a $2-$tree. Then $|\Ult(\mathbb{B})| \geq 2^{\aleph_0}$ where $\Ult(\mathbb{B})$ is the set of ultrafilters on $\mathbb{B}$. 
\end{fact} 

\begin{proof} For any path $\gamma$ in $T$ and let $A_{\gamma} = \{B \in T: B \in \gamma\}$. Clearly, $A_\gamma$ has the finite intersection property (since if $B,C \in A_\gamma$, then either $B \subset C$ or $C \subset B$) and so $A_\gamma$ can be extended to an ultrafilter over $\mathbb{B}$. This construction gives an injective map from paths in $T$ into ultrafilters on $\mathbb{B}$, proving the claim. 
\end{proof}

\begin{lemma}\label{twotree} Let $\mathbb{B}$ be a Boolean algebra on a set $X$. Assume that there exists a strongly continuous measure $\mu$ over $\mathbb{B}$. Then $\mathbb{B}$ has a $2-$tree.
\end{lemma} 

\begin{proof} Using $\mu$, we will build a $2-$tree. We build this tree in steps:

\noindent \textbf{Stage $0$:} Let $T_0 = \{X\}$. 

\noindent \textbf{Stage $n+1$:} We construct a tree of height $n+1$. Assume that $T_n$ is a (complete) binary tree of height $n$ such that for each $A \in T_{n}$, $\mu(A) > 0$. Assume furthermore that if $A, B \in T$ and $A \not \supset B$ and $B\not \supset A$, then $A \cap B = \emptyset$. We will construct $T_{n+1}$ by adding two children to each leaf. Let $\mathbb{L}_n$ be the collection of leaves on $T_n$. Let $\epsilon = \frac{\min\{\mu(L): L \in \mathbb{L}_n\}}{2}$. Since $\mu$ is strongly continuous, there exist $H_1,...,H_m \in \mathbb{B}$ such that $\mathbb{H}=\{H_1,...,H_m\}$ partitions $X$ and for each $j \leq m$, $\mu(H_j) < \epsilon$. Now fix a leaf $L_i$. Consider $L_i \cap \mathbb{H} = \{L_i \cap H : H \in \mathbb{H}\}$. We notice that $L_i \cap \mathbb{H}$ forms a partition of $L_i$. Therefore, 
\begin{equation*} 0<\mu(L_i) =\mu \Big( 
\bigcup_{K \in L_i \cap \mathbb{H}} K \Big) =\sum_{K \in L_i \cap \mathbb{H}} \mu(K).
\end{equation*}

\noindent Hence, there exists $K_r = L_i \cap H_{r} \in L_i \cap \mathbb{H}$ such that $\mu(K_r) > 0$. Furthermore,
\begin{equation*} \mu(K_r) = \mu(L_i \cap H_r) \leq \mu(H_r) < \epsilon \leq \frac{L_i}{2}.
\end{equation*}

\noindent Therefore there must exist some $K_l \in L_i \cap \mathbb{H}$ such that $K_l \neq K_r$ and $\mu(K_l) >0$. We now add $K_r, K_l$ as children to the leaf $L_i$. Let $T_{n+1}$ be the tree constructed after repeating this process for each $L \in \mathbb{L}_n$. Clearly, $T_{n+1}$ is a binary tree of height $n+1$ such that for each $A \in T_{n+1}$, $\mu(A) > 0$.

Now let $T = \bigcup_{n \geq 0} T_n$. $T$ is clearly a $2-$tree by construction.
\end{proof}

\begin{definition} Let $M_{\varphi}$ be the reduct of $M$ to the language $\mathcal{L}_{\varphi} = \{\varphi(x;y)\}$. A subset $N$ of $M$ is a $\varphi-$substructure of $M$, written $N \prec_{\varphi} M$, if the induced structure on $N$ (in the language $\mathcal{L}_{\varphi}$) is an elementary substructure of $M_{\varphi}$. 
\end{definition}

\begin{lemma}\label{lemma:main} Assume that $\varphi(x;y)$ is stable in $M$. Then there are no strongly continuous measures on $\Def_{\varphi}(M)$.
\end{lemma} 

\begin{proof} Assume that there exists a strongly continuous measure over $ \Def_{\varphi}(M)$. By Lemma \ref{twotree}, there exists a $2-$tree. Let $\mathbb{B}_0$ be the Boolean algebra generated by this $2-$tree. By Fact \ref{ultra}, $\mathbb{B}_0$ is a countable subalgebra of  $\Def_{\varphi}(M)$ such that  $\Ult(\mathbb{B}_{0}) \geq 2^{\aleph_0}$. Choose $C \subset M$ such that for each $B \in \mathbb{B}_0$, there exists $b_1,...,b_n$ in $C$ such that $B$ is an element of the boolean algebra generated by $\{\varphi(x;b_i): i \leq n\}$. Notice that since $\mathbb{B}$ is countable, we may choose $C$ to be countable. By the Downward L\"owenheim-Skolem theorem, there exists an $\mathcal{L}_\varphi-$structure $N$ such that $C \subseteq N$, $N \prec_{\varphi} M$, and $|N| = \aleph_0$. Then, 
\begin{equation*} 2^{\aleph_0}\leq |\Ult(\mathbb{B}_0)| \leq |\Ult(\Def_\varphi(C))| \leq |\Ult(\Def_\varphi(N))|=|S_\varphi(N)|.
\end{equation*}
However, since $\varphi(x;y)$ is stable in $M$, it is also stable in $N$. By Theorem \ref{B}, every $\varphi-$type over $N$ is definable by a $\varphi^{\opp}-$formula with parameters from $N$. 
Since $|N| = \aleph_0$, there are only countably many $\varphi^{\opp}-$formulas. Therefore, not every $\varphi-$type over $N$ is definable -- a contradiction. 
\end{proof} 

\begin{theorem}\label{Cl} Let $\varphi(x;y)$ be stable in $M$ and let $\mu \in \mathfrak{M}_{\varphi}(M)$. Then there exists an initial segment $I$ of $\mathbb{N}$ such that $\mu = \sum_{i \in I} r_i \delta_{p_i}$ where $p_i \in S_{\varphi}(M)$, $\sum_{i \in I} r_i = 1$, and each $r_i > 0$. Obviously, the statement also holds when $\varphi(x;y)$ is \emph{stable} (i.e., does not admit the $k$-order property for some $k$).
\end{theorem}

\begin{proof} Direct from the Sobczyk--Hammer Decomposition Theorem and Lemma \ref{lemma:main}.
\end{proof}

\subsection{The Morley product is commutative} We now aim to show that the Morley product commutes on appropriate pairs of measures. First, we need to appropriately define what we mean by the \emph{Morley product} in this context. To define it, we make some quick observations.

\begin{fact}\label{fact:function} Suppose that $X$ is a topological space and $Y$ is a dense subset of $X$. Let $f:Y \to Z$ be a map. If there exists some $\tilde{f}: X \to Z$ such that $\tilde{f}$ is continuous and $\tilde{f}|_{Y} = f$, then $\tilde{f}$ is the unique function with such property. 
\end{fact} 

\begin{proof} Clear via the net definition of continuity.
\end{proof} 

\begin{proposition}\label{prop:exist} Suppose that $\varphi(x;y)$ is stable in $M$ and $\mu \in \mathfrak{M}_{\varphi}(M)$. Consider the map $f_{\mu}^{\varphi}:\{\tp_{\varphi^{\opp}}(b/M): b \in M^{y}\} \to [0,1]$ via $f_{\mu}^{\varphi}(\tp_{\varphi^{\opp}}(b/M)) = \mu(\varphi(x,b))$. This map is well-defined and there exists a unique continuous function $F_{\mu}^{\varphi}:S_{\varphi^{\opp}}(M) \to [0,1]$ such that $F_{\mu}^{\varphi}|_{\{\tp_{\varphi^{\opp}}(b/M): b \in M^{y}\}} = f_{\mu}^{\varphi}$. 
\end{proposition} 

\begin{proof} By Theorem \ref{Cl}, $\mu = \sum_{i \in I} r_i \delta_{p_i}$ where each $p_i \in S_{\varphi}(M)$. We argue that the map $f_{\mu}^{\varphi}$ is well-defined. Notice that if $b \in M^{y}$, then  
\begin{equation*} 
\mu(\varphi(x,b)) = \sum_{i \in I} r_i [\delta_{p_i}(\varphi(x,b))] = \sum_{\substack{i \in I \\ M\models d_{p_i}^{\varphi}(b)}} r_i. 
\end{equation*} 
By Theorem \ref{B}, each formula $d_{p_i}^{\varphi}(y)$ is a $\varphi^{\opp}-$formula which implies that the value above only depends on the $\varphi^{\opp}-$type of $b$. Well-definedness follows.

Since $\{\tp_{\varphi^{\opp}}(b/M): b \in M^{y}\}$ is a dense subset of $S_{\varphi^{\opp}}(M)$, by Fact \ref{fact:function} it suffices to prove that there exists a continuous map from $S_{\varphi^{\opp}}(M)$ to $[0,1]$ which restricts to $f_{\mu}^{\varphi}$. We claim that $\sum_{i \in I} r_i \mathbf{1}_{[d_{p_i}^{\varphi}(y)]}$ is the appropriate map. We remark that we may view $\sum_{i \in I} r_i\mathbf{1}_{[d_{p_i}^{\varphi}(y)]}$ as a map from $S_{\varphi^{\opp}}(M)$ to $[0,1]$ since stability in $M$ implies that every formula of the form $d_{p_i}^{\varphi}(y)$ is a $\varphi^{\opp}-$formula (Theorem \ref{B}). 
\end{proof}

We may now define the \emph{Morley product} in this setting.

\begin{definition} Suppose that $\varphi(x;y)$ is stable in $M$. Let $\mu \in \mathfrak{M}_{\varphi}(M)$ and $\nu \in \mathfrak{M}_{\varphi^{\opp}}(M)$. We define the Morley product of $\mu$ with $\nu$, denoted $\mu_{x} \otimes \nu_{y}$, as follows:
\begin{equation*} 
(\mu \otimes \nu)(\varphi(x,y)) = \int_{S_{\varphi^{\opp}}(M)} F_{\mu}^{\varphi} d\tilde{\nu},
\end{equation*} 
where $F_{\mu}^{\varphi}$ is the function from Proposition \ref{prop:exist} and $\tilde{\nu}$ is the regular Borel probability measures corresponding to $\nu$. Likewise, since $\varphi(x;y)$ is stable in $M$ if and only if $\varphi^{\opp}(x;y)$ is stable in $M$, we may also define
\begin{equation*} 
(\nu \otimes \mu)(\varphi(x,y)) = \int_{S_{\varphi}(M)} F_{\nu}^{\varphi^{\opp}} d\tilde{\mu}, 
\end{equation*} 
with the obvious analogous definitions. 
\end{definition} 

\begin{remark}\label{remark:used} Since our definition of the Morley product is slightly non-standard, we are carful to make sure it resembles the normal Morley product on types. Suppose that $\varphi(x;y)$ is stable in $M$, let $p \in S_{\varphi}(M)$, $q \in S_{\varphi^{\opp}}(M)$, and fix $\mathcal{U}$ such that $M \prec \mathcal{U}$. Let $\hat{p} \in S_{\varphi}(\mathcal{U})$ be the unique $M$-definable extension of $p$ to $\mathcal{U}$. Then $(\delta_{p} \otimes \delta_{q})(\varphi(x,y)) = 1$ if and only if $\mathcal{U} \models \varphi(a,b)$ where $b \models q$ and $a \models \hat{p}|_{Mb}$. Indeed, consider the following sequence of bi-implications: 
\begin{align*} 
\left( \delta_{p} \otimes \delta_{q} \right)(\varphi(x,y)) = 1 &\Longleftrightarrow  \int_{S_{\varphi^{\opp}}(M)} \chi_{[d_{p}^{\varphi}(y)]}d\delta_{q} = 1 \Longleftrightarrow \delta_{q}(d_{p}^{\varphi}(y)) = 1\\
&\Longleftrightarrow d_{p}^{\varphi}(y) \in q  \Longleftrightarrow \mathcal{U} \models d_{p}^{\varphi}(b) \Longleftrightarrow \varphi(x,b) \in \hat{p} \Longleftrightarrow \mathcal{U} \models \varphi(a,b). 
\end{align*} 
\end{remark} 

\begin{theorem}\label{theorem:fubi} Suppose that $\varphi(x;y)$ is stable in $M$. Then 
\begin{equation*} 
(\mu \otimes \nu)(\varphi(x;y)) = (\nu \otimes \mu)(\varphi(x;y)). 
\end{equation*} 
\end{theorem} 
\begin{proof}
Consider the following sequence of equations:
\begin{align*}
(\mu \otimes \nu)(\varphi(x,y)) &= \int_{S_{\varphi^{\opp}}(M)} F_{\mu}^{\varphi}d\tilde{\nu} = \int_{S_{\varphi^{\opp}}(M)} \sum_{i \in I} r_i \mathbf{1}_{[d_{p_i}^{\varphi}(y)]}d\tilde{\nu} \\
&=  \sum_{i \in I} r_i \nu(d_{p_i}^{\varphi}(y)) =  \sum_{i \in I} r_i \sum_{j \in J} s_j \delta_{q_j}(d_{p_i}^{\varphi}(y)) \overset{(*)}{=}   \sum_{i \in I} \sum_{j \in J} r_is_j\delta_{p_i}(d_{q_j}^{\varphi^{\opp}}(x)).
\end{align*} 
Equation $(*)$ is justified by Theorem \ref{B}. A symmetric computation shows 
\begin{equation*} 
(\nu \otimes \mu)(\varphi(x,y)) = \sum_{i \in I} \sum_{j \in J} r_is_j\delta_{p_i}(d_{q_j}^{\varphi^{opp}}(x)), 
\end{equation*} 
completing the proof. 
\end{proof}

\subsection{Some functional analysis and double limits} By Theorem \ref{theorem:fubi}, we may define the following evaluation map, $E_{\varphi}$, on appropriate pairs of Keisler measures. We prove that if $\varphi(x;y)$ is stable in $M$, then $E_{\varphi}$ also has the double limit property. Our proof follows directly from classical results in functional analysis, namely Grothendieck's double limit theorem and the Krein-Smulian theorem. We first recall the definition of the evaluation map from the introduction. 

\begin{definition} Suppose that $\varphi(x,y)$ is stable in $M$. Then we define the map $E_{\varphi}:\mathfrak{M}_{\varphi}(M) \times \mathfrak{M}_{\varphi^{\opp}}(M) \to [0,1]$ via 
\begin{equation*} 
E_{\varphi}(\mu,\nu) = (\mu \otimes \nu)(\varphi(x,y)).
\end{equation*} 
By Theorem \ref{theorem:fubi}, $E_{\varphi}(\mu,\nu)$ is also equal to $(\nu \otimes \mu)(\varphi(x,y))$.  
\end{definition}

A proof of the following theorem can be found in \cite[Theorem 13.4]{JC}. 

\begin{theorem}[Krein-Smulian Theorem]\label{dl} If $Y$ is a Banach space and $K$ is weakly compact subset of $Y$, then the closed convex hull of $K$, denoted $\co(K)$, is weakly compact. The closed convex hull of $K$ is the intersection of all norm closed, convex subsets of $Y$ containing $K$. 
\end{theorem}

\begin{corollary} If $Y$ is a Banach space and $Z$ is a relatively weakly compact subset of $Y$ , then $\co(Z)$ is a weakly compact subset of $Y$. 
\end{corollary}

\begin{proof} Let $Z^{w}$ denote the weak closure of $Z$. Then $Z^{w}$ is weakly compact and so by the Krein-Smulian theorem, $\co(Z^{w})$ is weakly compact. Note that $\co(Z) \subseteq \co(Z^{w})$ and since $\co(Z)$ is a closed subset of a compact set, it is also compact. 
\end{proof} 

\begin{definition} Suppose that $X$ is a set. If we endow $X$ with the discrete topology and let $\mathcal{M}(X)$ be the collection of regular Borel probability measures on $X$, then we can consider the following convex collection of measures, 
\begin{equation*}
\conv(X) \colon = \left\{ \sum_{i \in I} r_i \delta_{x_i}: I \subseteq \mathbb{N}; x_i \in X; r_i \in \mathbb{R}_{\geq 0}; \sum_{i \in I} r_i = 1 \right\}.
\end{equation*} 
For simplicity of notation, we write $\sum_{i \in I} r_i \delta_{x_i}$ simply as $\sum_{i \in I} r_i x_i$.
\end{definition}

\begin{definition} Suppose that $X$ and $Y$ are sets and $f: X \times Y \to [0,1]$. Then we define $f_c:\conv(X) \times \conv(Y) \to [0,1]$ via $$ f_c\left(\sum_{i \in I} r_ix_i, \sum_{j \in J} s_jy_j \right) = \sum_{i \in I} \sum_{j \in J} r_is_jf(x_i,y_i).$$
\end{definition} 

We now prove the key combinatorial proposition via functional analysis.

\begin{proposition}\label{CDLT} Suppose that $X$ and $Y$ are sets and $f:X\times Y \to [0,1]$. Then $f$ has the double limit property if and only if $f_c:\conv(X) \times \conv(Y) \to [0,1]$ has the double limit property. 
\end{proposition} 

\begin{proof} 
If $f_c$ has the double limit property, then clearly $f$ has the double limit property. Therefore, we only need to prove the other direction.
Endow $Y$ with the discrete topology. Let $\mathbf{X} =\{f(a,y): a \in X\}$. It is obvious that  $\mathbf{X} \subset C_b(Y)$ and $\mathbf{X}$ is $||\cdot||_{\infty}-$bounded. By assumption, $f$ has the double limit property and by Theorem \ref{GR}, $\mathbf{X}$ is relatively weakly compact in $C_b(Y)$. 
By Corollary \ref{dl}, we have that $\co(\mathbf{X})$ is a weakly compact subset of $C_b(Y)$. Since $\co(\mathbf{X})$ is weakly compact in $C_b(Y)$, it is also relatively weakly compact, and so we can apply Theorem \ref{GR}. So, for any infinite sequences $g_i \in \co(\mathbf{X})$ and $b_j \in Y$

\begin{equation*}\lim_i \lim_j g_i(b_j) = \lim_j \lim_i g_i(b_j),
\end{equation*}
\noindent provided both limits exist. In particular, this implies that for $g_i := \sum_{\ell_i \in L_i} r_{\ell_i} f(a_{\ell_i},y)$,

\begin{equation*} 
\lim_i \lim_j f_c\left(\sum_{\ell_i \in L_i} r_{\ell_i} a_{\ell_i},b_j\right) = \lim_j \lim_i f_c\left(\sum_{\ell_i \in L_i} r_{\ell_i} a_{\ell_i},b_j\right),
\end{equation*} 
provided both limits exist. 

Notice that the computation above demonstrates that the map $f_c|_{\conv(X) \times Y}$ has the double limit property. Now consider $\conv(X)$ endowed with the discrete topology and let $\mathbf{Y} = \{f_c(x,b): b \in Y\}$.  It is clear that $\mathbf{Y} \subset C_b(\conv(X))$ and that $\mathbf{Y}$ is $||\cdot||_{\infty}-$bounded since each function is bounded by $1$. Since $f_c|_{\conv(X)\times Y}$ has the double limit property, we can again apply Theorem \ref{GR} and so $\mathbf{Y}$ is relatively weakly compact in $C_b(\conv(X))$. By the Corollary \ref{dl}, $\co(\mathbf{Y})$ is weakly compact in $C_b(\conv(X))$. By Theorem \ref{GR}, $f_c$ has the double limit property.
\end{proof}

\begin{corollary}\label{ML} If $\varphi(x;y)$ is stable in $M$, then the map $E_{\varphi}: \mathfrak{M}_{\varphi}(M) \times \mathfrak{M}_{\varphi^{*}}(M) \to [0,1]$ has the double limit property. 
\end{corollary} 

\begin{proof} By Fact \ref{Equiv}, the map 
$\chi_{\varphi}:S_{\varphi}(M) \times S_{\varphi^{\opp}}(M) \to \{0,1\}$
 has the double limit property. By Theorem \ref{Cl}, we have that 
$\mathfrak{M}_{\varphi}(M) = \conv(S_{\varphi}(M))$ and since $\varphi^{\opp}$ is also stable in $M$, 
$\mathfrak{M}_{\varphi^{\opp}}(M)= \conv(S_{\varphi^{\opp}})$. The computation in Theorem \ref{theorem:fubi} demonstrates that  $E_{\varphi} = (\chi_{\varphi})_{c}$. By Proposition \ref{CDLT}, $E_{\varphi}$ has the double limit property. 
\end{proof}

\section{Proper stability and the order property}

In this section, we work with \emph{honest-to-goodness} stable formulas and give a proof of an implicit theorem of Ben Yaacov and Keisler. Another proof of this theorem is given by Khanaki and  Pourmahdian using indiscernible arrays (see \cite[Theorem 3.11]{khanaki2024simple}). We show that if $\varphi(x;y)$ is stable, then the evaluation map $E_{\varphi}$ does not witness the continuous logic analogue of the order property. Throughout this section, we fix $\mathcal{L}-$structures $M$ and $\mathcal{U}$ such that $M \prec \mathcal{U}$ and $\mathcal{U}$ is a monster model. We let $T$ be the theory of $M$ in the language $\mathcal{L}$. We first show how to use the randomization to derive a proof. We then give another proof using the VC theorem. Given a theory $T$, we denote by $T^R$ its randomization. We refer the reader to Section 3.2 of \cite{conant2023generic} for background and notation regarding the randomization.

The following fact is due to Ben Yaacov and Keisler \cite[Theorem 5.14]{ben2009randomizations}.

\begin{fact}\label{fact:BK} Suppose that $\varphi(x;y)$ is a stable formula with respect to $T$. Then the randomized formula $\mathbb{E}[\varphi(x;y)]$ is stable (in the sense of continuous logic) with respect to $T^{R}$. In other words, if $N$ is a model of $T^{R}$ then for every $r \in (0,1)$ and $\epsilon > 0$, there exists some integer $n = n(\epsilon,r)$ such that there does not exist an array of elements $(\mathbf{a}_i,\mathbf{b}_j)_{(i,j) \in [n] \times [n]}$ from $N^{x} \times N^{y}$ such that 
\begin{equation*} 
\mathbb{E}[\varphi(\mathbf{a}_i,\mathbf{b}_j)] \geq r + \epsilon \text{ whenver } i \geq j, 
\end{equation*}  
and 
\begin{equation*} 
\mathbb{E}[\varphi(\mathbf{a}_i,\mathbf{b}_j)] \leq r \text{ whenver } i < j.
\end{equation*}  
Note that the integer $n$ does not depend on the choice of model, $N$. 
\end{fact} 

Using the above, it is easy to see that $E_{\varphi}$ also does not witness the continuous version of the order property. This follows from the observation that the randomization encodes the computations of the Morley product. 

\begin{proposition}\label{prop:order} Suppose that $\varphi(x;y)$ is stable. For every $r \in (0,1)$ and $\epsilon > 0$, there exists some integer $n = n(\epsilon, r)$ such that there does not exist an array of Keisler measures $(\mu_i,\nu_j)_{(i,j) \in [n] \times [n]}$ where $\mu_i \in \mathfrak{M}_{\varphi}(M)$ and $\nu_j \in \mathfrak{M}_{\varphi^{\opp}}(M)$ such that
\begin{equation*} 
(\mu_i \otimes \nu_j)(\varphi(x,y)) \geq r + \epsilon \hspace{.1in} \text{whenever $i \geq j$}, 
\end{equation*} 
and, 
\begin{equation*} 
(\mu_i \otimes \nu_j)(\varphi(x,y)) \leq r  \hspace{.1in} \text{whenever $i < j$}. 
\end{equation*}
\end{proposition} 

\begin{proof} Consider $[0,1]^{2}$ with the corresponding Lesbegue measure $L^{2}$ and the simple models of the randomization of $T$ relative to $[0,1]^{2}$, namely $M^{[0,1]^{2}}$ and $\mathcal{U}^{[0,1]^{2}}$. More explicitly, if $N$ is a model of $T$ then $N^{[0,1]}$ is the collection of measurable maps from $[0,1]^{2}$ to $N$ with finite image. It follows from quantifier elimination of $T^{R}$ that $M^{[0,1]^{2}} \prec \mathcal{U}^{[0,1]^{2}}$. If $\mu \in \mathfrak{M}_{\varphi}(M)$ and $\nu \in \mathfrak{M}_{\varphi^{\opp}}(M)$, then 
Theorem \ref{Cl} implies that $\mu = \sum_{k \in K } r_{k}\delta_{p_{k}}$ and $\nu = \sum_{w \in W } d_{w}\delta_{q_{w}}$ where $K$ and $W$ are initial segments of $\mathbb{N}$ and
\begin{enumerate} 
\item for each $k \in K$, $p_{k}$ is in $S_{\varphi}(M)$,
\item for each $w \in W$, $q_{w}$ is in $S_{\varphi^{\opp}}(M)$,
\item for each $k \in K$ and $w \in W$, $r_{k}$ and $d_{w}$ are positive real numbers,
\item $\sum_{k \in K} r_k = \sum_{w \in W} d_{w} = 1$. 
\end{enumerate} 

For each $q_{w}$, choose some $b_{w}$ in $\mathcal{U}^{y}$ such that $b_{w} \models q_{w}$. Let $\mathbf{b}_{\nu}:[0,1]^{2} \to \mathcal{U}$ via $\mathbf{b}_{\nu}((s,t)) = b_{w}$ 
whenever $s \in [\sum_{\ell=0}^{w-1} d_{\ell}, \sum_{\ell=0}^{w} d_{\ell})$ with the convention that $\sum_{\ell = 0}^{-1} d_{\ell} = 0$. For each $p_{k}$, choose some $a_{k}$ in $\mathcal{U}^{x}$ such that $a_{k} \models \hat{p}_{k}|_{M(b_{w})_{w \in W}}$ where $\hat{p}_{k}$ is the unique $M$-definable extension of $p$ in $S_{\varphi}(\mathcal{U})$. Let $\mathbf{a}_{\mu}:[0,1]^{2} \to \mathcal{U}$ via $\mathbf{a}_{\mu}((s,t)) = a_{k}$ when $t \in [\sum_{\ell=0}^{k-1} r_{\ell}, \sum_{\ell=0}^{k} r_{\ell})$ again with the convention that $\sum_{\ell = 0}^{-1} r_{\ell} = 0$. In the following computations, if $(a,b) \in \mathcal{U}^{x} \times \mathcal{U}^{y}$, we let $\varphi(a,b) = 1$ if $\mathcal{U} \models \varphi(a,b)$ and $0$ otherwise. Then 
\begin{align*} 
(\mu \otimes \nu)(\varphi(x,y)) &\overset{(a)}{=} \sum_{k \in K} \sum_{w \in W} r_k d_w (\delta_{q_w}(d_{p_k}^{\varphi}(y))) \overset{(b)}{=} \sum_{k \in K} \sum_{w \in W} r_k d_w \varphi(a_k,b_w) \\
&=\int_{(s,t) \in [0,1]^{2}} \varphi(\mathbf{a}_{\mu}(s,t)),\mathbf{b}_{\nu}(s,t))dL^{2} = \mathbb{E}[\varphi(\mathbf{a}_{\mu},\mathbf{b}_{\nu})]. 
\end{align*} 
Equation $(a)$ is derived in the proof of Theorem \ref{Cl}. Equation $(b)$ follows from Remark \ref{remark:used}. Thus, if the statement is false, then $\mathbb{E}[\varphi(x,y)]$ witnesses the continuous logic version of the order property. This contradicts Fact \ref{fact:BK}. 
\end{proof} 

We now work to give a second proof of Proposition \ref{prop:order} via the VC theorem. The statement of the VC theorem given below is much weaker than the general statement, but it is all that we need. 

\begin{theorem}[VC--theorem] Suppose that $X$ is a set $\mathcal{F}$ is a collection of subsets. 
Suppose that the VC-dimension of the class $\mathcal{F}$ is bounded by $d$. 
Then for every $\epsilon > 0$, there exists an integer $n = n(\epsilon,d)$ such that for every atomic measure $\mu$ on $X$ 
(i.e., $\mu = \sum_{i \in I} r_i \delta_{x_i}$ where $I \subseteq \mathbb{N}$),
 there exists $a_1,...,a_n \in X$ such that for any $F \in \mathcal{F}$, 
\begin{equation*} 
\sup_{F \in \mathcal{F}} |\mu(F) - \Av(a_1,...,a_n)(F)| < \epsilon. 
\end{equation*} 
We remark that $n$ does not depend on the choice of measure. 
\end{theorem}

\begin{lemma}\label{prop:approx} Suppose that $\varphi(x,y)$ is stable. For every $\epsilon > 0$ there exists some natural number $N = N(\epsilon)$ such that for every $\mu \in \mathfrak{M}_{\varphi}(M)$ there exists $a_1,...,a_N \in M$ such that for any $b \in M^{y}$, 
\begin{equation*} 
|\mu(\varphi(x,b)) - \Av(a_1,....,a_{N})(\varphi(x,b))| < \epsilon. 
\end{equation*} 
\end{lemma}

\begin{proof} There are several ways too see this. By Corollary \ref{Cl}, we may write $\mu = \sum_{i=0}^{\omega} r_i \delta_{p_i}$ where each $p_i$ is in $S_{\varphi}(M)$. Let $\mathcal{U}$ be a monster model such that $M \prec \mathcal{U}$ and consider the measure $\hat{\mu} \in \mathfrak{M}_{\varphi}(\mathcal{U})$ given by $\sum_{i = 0}^{\omega} r_i \delta_{\hat{p}_i}$ where $\hat{p}_i$ is the unique global $M$-definable extension of $p$ to $\mathcal{U}$. For each $i \in \omega$, we have that $\hat{p}_i$ is both definable over $M$ and finitely satisfiable in $M$ (\cite[Proposition 2.3]{Pillay}). As consequence, the measure $\hat{\mu}$ is finitely satisfiable and $\varphi-$definable over $M$ (for appropriate definitions in this context, see \cite[Section 6]{gannon2022sequential}). If $\varphi(x;y)$ is stable, it is NIP and so an application of \cite[Theorem 6.4]{gannon2022sequential} implies that for every $\epsilon > 0$, there exists $a_1,...,a_n \in M^{x}$ such that 
\begin{equation*} 
\sup_{b \in \mathcal{U}^{y}}| \mu(\varphi(x,b)) - \Av(\bar{a})(\varphi(x,b))| < \epsilon.  
\end{equation*} 
An application of the VC theorem gives uniform bounds. 
\end{proof} 

\begin{altproof} Suppose not. Then there exists $r \in (0,1)$ and $\epsilon > 0$ and sequence $(\mu_i,\nu_j)_{(i,j) \in \omega \times \omega}$ which witnesses the $(r,\epsilon)-$order property. We now construct a discrete formula (which is a Boolean combination of $\varphi(x;y)$) which witnesses the order property -- and since stable formulas are closed under Boolean combinations, we have a contradiction. By Proposition \ref{prop:approx}, there exists a natural number $N$ such that 
\begin{enumerate} 
\item For every $\mu \in \mathfrak{M}_{\varphi}(M)$, there exists $a_1,\dots,a_N \in M^{x}$ such that 
\begin{equation*} 
\sup_{b \in M^{y}} |\mu(\varphi(x,b)) - \Av(\bar{a})(\varphi(x,b))| < \frac{\epsilon}{16}. 
\end{equation*}
For each $i \in \omega$, let $\bar{a}_i = a^i_1,\dots,a^i_n$ witness the above equation for $\mu_i$. 
\item For every $\nu \in \mathfrak{M}_{\varphi^{\opp}}(M)$, there exists $b_1,\dots,b_N \in M^{y}$ such that 
\begin{equation*} 
\sup_{a \in M^{x}} |\nu(\varphi(a,y)) - \Av(\bar{b})(\varphi(a,y))| < \frac{\epsilon}{16}. 
\end{equation*}
For each $j \in \omega$, let $\bar{b}_j = b^j_1,\dots,b^j_N$ witness the above equation for $\nu_j$. 
\end{enumerate} 
Consider the formula given by, 
\begin{equation*} 
\theta(x_1,...,x_N,y_1,...,y_N) := \bigvee_{\substack{A \times B \subseteq [N] \times [N] \\ \frac{|A \times B|}{N^{2}} > r + \frac{\epsilon}{2}}} \left( \bigwedge_{(i,j) \in A \times B} \varphi(x_i,y_i) \right).
\end{equation*} 
We claim that $\theta(\bar{x},\bar{y})$ is unstable. Notice that 
\begin{equation*} 
M \models \theta(\bar{a}_i,\bar{b}_j) \Longrightarrow M \models \bigvee_{\substack{A \times B \subseteq [N] \times [N] \\ \frac{|A \times B|}{N^{2}} > r + \frac{\epsilon}{2}}} \left( \bigwedge_{(\ell,k) \in A \times B} \varphi(a^{i}_\ell,b^{j}_k) \right) \Longrightarrow (\Av(\bar{a}_i) \otimes \Av(\bar{b}_j))(\varphi(x,y)) > r + \frac{\epsilon}{2},
\end{equation*} 
and likewise, 
\begin{equation*}
M \models \neg \theta(\bar{a}_i,\bar{b}_j) \Longrightarrow (\Av(\bar{a}_i) \otimes \Av(\bar{b}_j))(\varphi(x,y)) \leq r + \frac{\epsilon}{2}. 
\end{equation*} 
Moreover, if $i < j$, then
\begin{align*} 
r \geq (\mu_i \otimes \nu_j) (\varphi(x,y)) &\approx_{\epsilon/16} (\Av(\bar{a}_i) \otimes \nu_j) (\varphi(x,y))  \\ 
&= (\nu_j \otimes \Av(\bar{a}_i)) (\varphi(x,y)) \approx_{\epsilon/16} (\Av(\bar{b}_j) \otimes \Av(\bar{a}_i)) (\varphi(x,y)),
\end{align*} 
and likewise, if $i \geq j$, 
\begin{align*} 
r + \epsilon \leq (\mu_i \otimes \nu_j) (\varphi(x,y)) &\approx_{\epsilon/16} (\Av(\bar{a}_i) \otimes \nu_j) (\varphi(x,y))  \\ 
&= (\nu_j \otimes \Av(\bar{a}_i)) (\varphi(x,y)) \approx_{\epsilon/16} (\Av(\bar{b}_j) \otimes \Av(\bar{a}_i)) (\varphi(x,y)),
\end{align*} 
Hence, if $i < j$ and $\models \theta(\bar{a}_i,\bar{b}_i)$, then $\Av(\bar{a}_i) \otimes \Av(\bar{b}_j)(\varphi(\bar{x},\bar{y}))$ is greater than $r + \frac{\epsilon}{2}$ (by witnessing $\theta$) and less than $r + \frac{\epsilon}{8}$ (by the above implication) -- a contradiction. Hence, if $i < j$, then $\models \neg \theta(\bar{a}_i,\bar{b}_i)$. A similar argument shows that if $i \geq j$ then $\theta(\bar{a}_i,\bar{b}_j)$ must hold. Thus $\theta(\bar{x},\bar{y})$, a Boolean combination of $\varphi(x;y)$, is unstable -- a contradiction. 
\end{altproof} 

\subsection*{Acknowledgements} We thank Sergei Starchenko and Gabriel Conant for helpful discussions during the writing of the initial research note. We thank the country of Kazakhstan for their warm welcome and hospitality. 

\bibliographystyle{plain}
\bibliography{refs}
\end{document}